\documentclass{article}
\usepackage{epsfig}
\usepackage{amsmath}
\def\bbbr{{\rm I\!R}} 

\newtheorem{lemma}{Lemma}
\newtheorem{theorem}{Theorem}

\newtheorem{claim}{Claim}

\newtheorem{cor}{Corollary}

\newcommand{\bn}{\begin{displaymath}}            
\newcommand{\en}{\end{displaymath}}
\newcommand{\bq}{\begin{equation}}               
\newcommand{\eq}{\end{equation}}
\newenvironment{proof}{
\par
\noindent {\bf Proof.}\rm}%
{\mbox{}\hfill\rule{0.5em}{0.809em}\par}
\newcommand{\qed}{\hfill\rule{0.5em}{0.809em}}

\begin{document}

\title{Circular flow number of highly edge connected  signed graphs}
\author{Xuding Zhu \thanks{Department of Mathematics, Zhejiang Normal University, China. Grant numbers:
  NSF11171310 and ZJNSF  Z6110786.  Email: xudingzhu@gmail.com. }
       }

\date{2012.11.11}

\maketitle

\begin{abstract}
This paper proves that for any positive integer $k$, every essentially $(2k+1)$-unbalanced
$(12k-1)$-edge connected signed graph has circular flow number at most $2+\frac 1k$.
\end{abstract}

\section{Introduction}

Suppose $G$ is a graph. A {\em circulation} in $G$ is  an orientation $D$ of $G$ together with a mapping $f: E(G) \to \bbbr$.
A circulation in $G$ can be denoted by a pair $(D, f)$. However, for simplicity, we usually denote it by $f$, and
 call $D$ the {\em orientation   associated with   $f$}.
The {\em boundary} of a circulation $ f$ is the map
$\partial f: V(G) \to \bbbr$ defined as
$$\partial f(v) = \sum_{e \in E^+_D(v)} f(e) - \sum_{e \in E^-_D(v)}f(e). \eqno(1)$$
Here $E^+_D(v)$ (resp. $E^-_D(v)$) is the set of directed edges in $D$ of the form $(v,u)$ (resp. of the form $(u, v)$).

A {\em flow in $G$} is a circulation in $G$ with $\partial f =0$.
If $r$ is a real number and $f$ is a flow with
$1 \le |f(e)| \le r-1$ for every edge $e$, then $f$ is called a {\em circular $r$-flow} in $G$.
The {\em circular flow number} $\Phi_c(G)$ of $G$ is the least $r$ such that $G$   admits a circular
$r$-flow. If $r$ is an integer and $1 \le |f(e)| \le r-1$ are integers   then $f$ is called a {\em nowhere zero $r$-flow}.
The {\em flow number} $\Phi(G)$ of $G$ is the least integer $r$ such that $G$ admits a nowhere zero $r$-flow.
It is known \cite{GTZ98} that $\Phi(G) = \lceil \Phi_c(G) \rceil$ for any bridgeless graph $G$.

Integer flow was originally introduced by Tutte \cite{Tutte1949, Tutte1954} as a generalization of map colouring.
Tutte proposed the following three conjectures that motivated most of the studies on integer flow in graphs.
\begin{itemize}
\item {\bf $5$-Flow Conjecture}: Every bridgeless graph admits a nowhere zero $5$-flow.
\item {\bf $4$-Flow Conjecture}: Every bridgeless graph with no Petersen minor admits a nowhere zero $4$-flow.
\item {\bf  $3$-Flow Conjecture}: Every $4$-edge connected graph admits a nowhere zero $3$-flow.
\end{itemize}

The concept of circular flow number was introduced in \cite{GTZ98} in 1998 as the dual of the circular chromatic number
(cf. \cite{Survey01}),
and as a refinement of the flow number. However, in early 1980's,
Jaeger \cite{JAE81} already studied circular flow in graphs and proposed a conjecture which is equivalent to the following:

\begin{itemize}
\item {\bf $(2+\frac 1k)$-flow conjecture}: For any positive integer $k$, if $G$ is $4k$-edge connected, then $\Phi_c(G) \le 2 + \frac 1k$.
\end{itemize}

Jaeger's conjecture is very strong. The $k=1$ case is the $3$-flow conjecture,
and the $k=2$ case implies the $5$-flow conjecture.

All the above conjectures remain open.  Recently,    Thomassen \cite{Thomassen2012}  made a breakthrough
in the study of $3$-flow conjecture   by proving  that every $8$-edge connected graph
 admits a nowhere zero
$3$-flow. Moreover, for $k \ge 1$, every $(8k^2+10k+3)$-edge connected graph
has circular flow number at most $2 + \frac 1k$.  This result is improved by
Lov\'{a}sz,  Thomassen,  Wu and  Zhang in \cite{LTWZ2012},
where it is proved that for any positive integer $k$, if a graph
$G$ has odd edge connectivity at least $6k+1$, then $\Phi_c(G) \le 2 + \frac 1k$.

This paper proves an analog of this result for signed graphs.

A {\em signed graph} is a pair $(G, \sigma)$, where $G$ is a graph and $\sigma: E(G) \to \{1,-1\}$
assigns to each edge a sign: an edge
$e$ is either  a {\em positive edge} (i.e., $\sigma(e)=1$) or a {\em negative edge} (i.e., $\sigma(e)=-1$).
An {\em orientation }  $\tau$  of $(G, \sigma)$ assigns "orientations" to the edges of $G$ as follows: if $e=xy$ is a positive edge, then
the edge is oriented either from $x$ to $y$ or from $y$ to $x$.
In the former case, $e \in E^+_{\tau}(x) \cap E^-_{\tau}(y)$,
and in the later case, $e \in E^-_{\tau}(x) \cap E^+_{\tau}(y)$.
If $e=xy$ is a negative edge, then
the edge is oriented either from both $x$ and $y$ or towards both $x$ and $y$. In the former case, $e \in E^+_{\tau}(x) \cap E^+_{\tau}(y)$
and $e$ is called a {\em sink edge}.
In the later case, $e \in E^-_{\tau}(x) \cap E^-_{\tau}(y)$ and $e$ is called a {\em source edge}.
An orientation $\tau$ of $(G,\sigma)$ may be  viewed as a mapping which assigns to each positive edge one of its end vertices as the {\em head} of the directed edge,
and labels each negative edge either as a source edge or as a sink edge.
If the orientation is clear from the context, we write $E^+(x)$ for $E^+_{\tau}(x)$, and etc.
An oriented signed graph (i.e., a signed graph plus an orientation)  is  called a {\em bidirected graph}.

If all the edges of $G$ are positive, then an orientation of $(G,\sigma)$ is a directed graph.
So  the concept of signed graphs is a generalization  of graphs,
and bidirected graphs is a generalization of  digraphs.
All the concepts concerning flow in  graphs
can be naturally extended to signed graphs.
For example,
a circulation in a signed graph $(G, \sigma)$ is an orientation $\tau$ of $(G, \sigma)$ together with
 a mapping $f: E(G) \to \bbbr$. Similarly, we usually denote a circulation in $(G,\sigma)$ by $f$, and call
 $\tau$ the orientation of $(G,\sigma)$ associated with $f$.
 The boundary $\partial f$ of a circulation of $(G, \sigma)$ is defined in the same way as in (1).
The concepts of flow, circular $r$-flow, nowhere zero $r$-flow, circular flow number and flow number are extended to signed graphs
in the same way.
In case a signed graph $(G,\sigma)$ does not admit a nowhere zero
$k$-flow for any $k$, then let $\Phi(G,\sigma) = \Phi_c(G,\sigma) = \infty$.

For every question concerning  flow in graphs, one can ask the corresponding question
for signed graphs. However, less results  are known for flow in signed graphs. There are also less conjectures
for flow in signed graphs. The only well-known conjecture for flow in signed graphs is the following
conjecture proposed by Bouchet \cite{Bou83}.

\begin{itemize}
\item
{\bf $6$-Flow Conjecture}: If a signed graph admits a nowhere zero  flow, then it admits a nowhere zero $6$-flow.
\end{itemize}

Bouchet \cite{Bou83} proved that if a signed graph admits a nowhere zero  flow, then its flow number is at most $216$.
Z\'{y}ka \cite{Zyka87} improved the upper bound  to $30$.
Khelladi \cite{Khell87} proved that for $4$-edge connected graphs, the upper bound can be reduced to $18$, and Xu and Zhang \cite{XuZhang05}
proved that  for $6$-edge connected graphs, the upper bound can be reduced to $6$.

The above mentioned results show that
highly edge connected graphs have circular flow number close to $2$.
One naturally wonder if a similar result holds for signed graphs.
By applying the main result in \cite{LTWZ2012}, we show that  the answer to  this question is positive,
provided that an additional minor necessary condition is satisfied.

If $(G, \sigma)$ is a signed graph and $v$ is a vertex of $G$, then by a {\em switching} at $v$, we obtain another
signed graph $(G, \sigma')$, where $\sigma'(e) = - \sigma(e)$ if $e \in E(v)$ and $\sigma'(e) = \sigma(e)$ otherwise.
Here $E(v)$ is the set of edges incident to $v$.
We say two signed graphs $(G, \sigma)$ and $(G, \sigma')$ with the same underlying graph $G$ are {\em equivalent} if
one can be obtained from the other by a sequence of switchings.
It is easy to see that equivalent signed graphs have the same flow number and the same  circular flow number.
We say a signed graph $(G, \sigma)$ is  {\em  $(2k+1)$-unbalanaced } if every signed graph equivalent to
$(G, \sigma)$  has   at least $2k+1$ negative edges, and say $(G, \sigma)$ is {\em essentially $(2k+1)$-unbalanced} if
every signed graph equivalent to
$(G, \sigma)$  has either an even number of negative edges or  at least $2k+1$ negative edges.
For a signed graph to have circular flow number at most $2+ \frac 1k$, being highly edge connected is not sufficient.
For example,   a highly edge connected signed graph may have exactly one negative edge. In this case
it does not admit a nowhere zero flow.
It is also easy to verify that if a signed graph has exactly $2k+1$ negative edges, then its circular flow number is at least
$2+ \frac 1k$.
So the "correct" question is whether an essentially $(2k+1)$-unbalanced (for some large integer $k$) highly edge connected signed graph
has circular flow number close to $2$.
Raspaud and Zhu \cite{RZ2011} proved that  $6$-edge connected $3$-unbalanced signed graph
has flow number at most $4$ and has circular flow number strictly less than
$4$.  In this paper, we  prove the following result.

\begin{theorem}
\label{main}
Let $k$ be a positive integer.  If a signed graph $(G, \sigma)$ is $(12k-1)$-edge connected and essentially $(2k+1)$-unbalanced, then
$\Phi_c(G,\sigma) \le 2 + \frac 1k$.
\end{theorem}

\section{Modulo $(2k+1)$-orientations and $Z_{2k+1}$-flow}

Given an orientation $\tau$ of a signed graph $(G, \sigma)$,
the out-degree and in-degree of each vertex $x$ is defined in the same way as for digraphs:
$d^+_{\tau}(x) = |E^+_{\tau}(x)|$ and $d^-_{\tau}(x) = |E^-_{\tau}(x)|$.

Fix a positive integer $k$. Given a mapping $\beta: V(G) \to Z_{2k+1}$,
an orientation $ \tau$ is called a $\beta$-orientation if for every vertex $x$ of $G$,
$$d^+_{\tau}(x)- d^-_{\tau}(x)\equiv \beta(x) \pmod{2k+1}.$$
If $\beta(x) = 0$, then a $\beta$-orientation of $(G, \sigma)$ is called   a {\em modulo $(2k+1)$-orientation}.

A mapping $\beta: V(G) \to Z_{2k+1}$  is called a $Z_{2k+1}$-boundary if $\sum_{x \in V}\beta(x) \equiv 0 \pmod{2k+1}$.
Note that for any orientation $D$ of an ordinary graph $G$, $\sum_{x \in V(G)} (d^+_D(x)-d^-_D(x)) = 0$.
Thus for a graph $G$ to have a $\beta$-orientation, one necessary condition is that   $\beta$ be a $Z_{2k+1}$-boundary.
The following result, proved in \cite{LTWZ2012}, says that if $G$ is highly edge connected, then the converse is also true.

\begin{theorem}
\label{LTWZ}\cite{LTWZ2012}
For any positive integer $k$, every $6k$-edge connected graph has a $\beta$-orientation for any $Z_{2k+1}$-boundary $\beta$ of $G$.
\end{theorem}

First we prove an analog of Theorem \ref{LTWZ} for signed graphs. Note that for an orientation $\tau$ of a signed graph $(G,\sigma)$,
the summation $\sum_{x \in V(G)} (d^+_{\tau}(x)-d^-_{\tau}(x)) $ is not necessarily $0$.

\begin{theorem}
\label{mod-o}
For any positive integer $k$, for any mapping $\beta: V(G) \to Z_{2k+1}$, every $(2k+1)$-unbalanced
$(12k-1)$-edge connected
signed graph $(G,\sigma)$
has a $\beta$-orientation.
\end{theorem}
\begin{proof}
Assume $(G,\sigma)$ is a $(2k+1)$-unbalanced
$(12k-1)$-edge connected
signed graph, and $\beta$ is a mapping from $V(G)$ to $Z_{2k+1}$.
We may assume $(G, \sigma)$ has  the minimum number of negative edges among all signed graphs
equivalent to $(G, \sigma)$. Otherwise let $(G, \sigma')$ be the one with minimum number of negative edges.
Assume $(G, \sigma')$ is obtained from $(G,\sigma)$ by switching vertices in $X$. Let $\beta': V(G) \to Z_{2k+1}$ be defined
as
\begin{equation*}
 \beta' (v)=\begin{cases}
 -\beta   (v),  &  \text{if $v \in X$,}
\\
\beta (v),  & \text{otherwise.}
\end{cases}
\end{equation*}
We shall show that $(G, \sigma')$ has a $\beta'$-orientation, which is
equivalent to  $(G, \sigma)$ has a $\beta$-orientation.

Fix a mapping $\beta: V \to Z_{2k+1}$.
Let $0 \le t \le 2k$ be the integer such that
$$  \sum_{x \in V(G)}\beta(x) \equiv 2t \pmod{2k+1}.$$
Let $Q$ be the set of negative edges in $(G, \sigma)$.
If $t \equiv |Q| \pmod{2}$, then let $\tau$ be an orientation of $Q$ with
$(|Q|+t)/2$  sink edges, and  $(|Q|-t)/2$  source edges.
Let $$\beta'(x) \equiv \beta(x) - (d^+_{\tau}(x)-d^-_{\tau}(x)) \pmod{2k+1}.$$ Since each sink edge (resp. source edge) contributes $2$ (resp. $-2$)
to the summation $\sum_{x \in V(G)} (d^+_{\tau}(x)-d^-_{\tau}(x))$, we conclude that
\begin{eqnarray*}
\sum_{x \in V(G)} (d^+_{\tau}(x)-d^-_{\tau}(x)) &=&  (|Q|+t)-(|Q|-t)  \\
&=& 2t   \equiv \sum_{x \in V(G)} \beta(x) \pmod{2k+1}.
\end{eqnarray*}
Thus $$\sum_{x \in V(G)} \beta'(x) \equiv \sum_{x \in V(G)} \beta(x) - \sum_{x \in V(G)} (d^+_{\tau}(x)-d^-_{\tau}(x)) \pmod{2k+1} \equiv 0 \pmod{2k+1}, $$ i.e., $\beta'$ is a $Z_{2k+1}$-boundary  of $G$.

If $t \equiv |Q|+1 \pmod{2}$, then let $\tau$ be an orientation of $Q$ with
$(|Q|+t-2k-1)/2$ sink edges and $(|Q|+2k+1-t)/2$   source edges.
The same calculation shows that  $\beta'$ is a $Z_{2k+1}$-boundary  of $G$.

Let $R$   be the subgraph of $G$ induced by positive   edges of $(G, \sigma)$.

\begin{claim}
\label{clm1}
The graph $R$ is $6k$-edge connected.
\end{claim}
\begin{proof}
Assume to the contrary that $R$ has an edge cut $E_R[X, \bar{X}]$ of size at most $6k-1$.
Since $G$ is $(12k-1)$-edge connected, we have $|E_Q[X, \bar{X}]| \ge 6k > |E_R[X, \bar{X}]|$.
Let $(G, \sigma')$ obtained from $(G, \sigma)$  by switching
at all vertices in $X$. Then  $\sigma'(e) = -\sigma(e)$ if $e \in E_G[X, \bar{X}]$
and $\sigma'(e) = \sigma(e)$ otherwise. Thus $(G, \sigma')$ has less  negative edges than $(G, \sigma)$,
contrary to our choice of $(G, \sigma)$.
\end{proof}
\bigskip

By Theorem \ref{LTWZ}, $R$ has a $\beta'$-orientation $D$,
i.e., $d^+_D(x) - d^-_D(x) \equiv \beta'(x) \pmod{2k+1}$ for every vertex $x$.
The union of this orientation of $R$ and the previously chosen orientation $\tau$ of $Q$
is an orientation $\tau'$ of  $(G,\sigma)$, with
$d^+_{\tau'}(x) = d^+_{\tau}(x) + d^+_D(x)$ and $d^-_{\tau'}(x) = d^-_{\tau}(x)+d^-_D(x)$. Hence for every vertex $x$,
\begin{eqnarray*}
d^+_{\tau'}(x) - d^-_{\tau'}(x) &=& (d^+_D(x)-d^-_D(x))+  (d^+_{\tau}(x)-d^-_{\tau}(x)) \\
&\equiv & \beta'(x) +  (d^+_{\tau}(x)-d^-_{\tau}(x)) \pmod{2k+1} \\
&\equiv&  \beta(x) \pmod{2k+1}
\end{eqnarray*}
I.e., $\tau'$ is a $\beta$-orientation of $(G,\sigma)$.
\end{proof}

For   Theorem \ref{mod-o}, the condition that $(G, \sigma)$ be $(2k+1)$-unbalanced is needed.
For example, if $\beta$ is not a $Z_{2k+1}$-boundary, and $(G, \sigma)$ has no negative edges, then
$(G, \sigma)$ cannot have a $\beta$-orientation. For the proof above, if $|Q| < 2k+1$, then
 the numbers $|Q|-t$ and $|Q|+t-(2k+1)$ appeared in the proof might be negative.
On the other hand, if we restrict to modulo $(2k+1)$-orientations, this condition can be slightly weakened.
The following Theorem can be proved in the same way as Theorem \ref{mod-o}.

\begin{theorem}
\label{mod-oo}
For any positive integer $k$,  every $(12k-1)$-edge connected,
essentially $(2k+1)$-unbalanced signed graph $(G,\sigma)$
has a modulo $(2k+1)$-orientation. \qed
\end{theorem}

Assume $(G,\sigma)$ is a signed graph and $A$ is an abelian group.
An $A$-circulation of $(G,\tau)$ is an orientation $\tau$ of $(G, \sigma)$
together with  a mapping
$f: E(G) \to A$. The boundary of an $A$-circulation $f$ of $(G,\sigma)$ is defined in
the same way as before, i.e., $\partial f(x) = \sum_{e \in E^+(x)}f(e) - \sum_{e \in E^-(x)}f(e)$,
where the summation is the group operation.
Similarly, an $A$-flow of $(G,\sigma)$ is an $A$-circulation $f$ with $\partial f(x)=0$ for every vertex $x$.

Let $k$ be a positive integer. We consider the group $Z_{2k+1}$.
A {\em special $Z_{2k+1}$-circulation} (resp.
{\em special $Z_{2k+1}$-flow}) in a signed graph $(G, \sigma)$ is a $Z_{2k+1}$-circulation
(resp. a $Z_{2k+1}$-flow)  $f$ with
$f(e) \in \{k, k+1\}$ for every edge $e$.

\begin{cor}
\label{z2k+1}
For any positive integer $k$, every $(12k-1)$-edge connected,
essentially $(2k+1)$-unbalanced signed graph $(G,\sigma)$
  admits a special $Z_{2k+1}$-flow.
\end{cor}
\begin{proof}
By Theorem \ref{mod-oo}, $(G, \sigma)$  has a modulo $(2k+1)$-orientation $\tau$.
Let $f(e) = k$ for all edges $e$ of $G$. Then $f$ is a special $Z_{2k+1}$-flow in $(G, \sigma)$.
\end{proof}

\section{Circular flow number}

Assume $p \ge 2q$ are positive integers. A $(p, q)$-flow in a signed graph $(G, \sigma)$ is an integer flow $f$ with
$f(e) \in \{q, q+1, \ldots, p-q\}$  for every edge $e$. If $(G. \sigma)$ admits a  $(p,q)$-flow $f$, then
 $g(e) = \frac{f(e)}{q}$ is a circular $p/q$-flow in  $(G,\sigma)$. Hence $\Phi_c(G, \sigma) \le p/q$.
(The converse is also true: if $\Phi_c(G, \sigma) \le p/q$, then $(G,\sigma)$ admits a $(p, q)$-flow. But we shall not use that.)
Thus to prove $\Phi_c(G, \sigma) \le 2 + \frac 1k$, it suffices to prove that $(G, \sigma)$ admits a
$(2k+1,k)$-flow.

It is well-known (cf. \cite{CQbook}) that  an ordinary graph $G$ admits a $(2k+1,k)$-flow if and only if
$G$ admits a special $Z_{2k+1}$-flow. However, this is not the case for signed graphs. It was proved in \cite{XuZhang05} that
if $(G, \sigma)$ is a cubic signed graph which admits a special $Z_3$-flow, then $(G, \sigma)$ admits a
nowhere zero $3$-flow (or equivalently a $(3,1)$-flow) if and only if $G$ has a perfect matching.
Nevertheless, we shall prove that the condition of Corollary \ref{z2k+1} implies that $(G, \sigma)$ admits a $(2k+1,k)$-flow
and hence $(G, \sigma)$ has circular flow number at most $2 + \frac 1k$.

\begin{theorem}
\label{main1}
If a signed graph $(G,\sigma)$ is
$(12k-1)$-edge connected and essentially $(2k+1)$-unbalanced, then $(G,\sigma)$ admits a $(2k+1,k)$-flow.
\end{theorem}
\begin{proof}
Assume $(G, \sigma)$ is $(12k-1)$-edge connected and essentially $(2k+1)$-unbalanced.
We assume that $(G, \sigma)$ has the minimum number of negative edges among all signed graphs equivalent to $(G,\sigma)$.
Let $R$ and $Q$ be the subgraphs of $(G,\sigma)$ induced by the set of positive edges and by the set of negative edges, respectively.

We shall construct a $(2k+1, k)$-flow in $(G, \sigma)$. This is done in two steps.
In the first step, we construct a special $Z_{2k+1}$-circulation $f$
in the subgraph $Q$. In the second step, we construct a special $Z_{2k+1}$-circulation
$g$ in $R$, so that $f+g$ is a $(2k+1, k)$-flow. In taking the sum $f+g$,
we view $f$ (resp. $g$)  as a circulation in $(G,\sigma)$ with $f(e)=0$ for every positive edge $e$ (resp. $g(e)=0$ for every
negative edge $e$).

Given  a special $Z_{2k+1}$-circulation $f$ in $Q$.
For a subset $X$ of $V(G)$. Let $E_G[X,   \bar{X}]$
is the set of edges in $G$ with one end vertex in $X$ and the other in
$\bar{X}=V \setminus X$ (we write $E[X, \bar{X}]$ for short, if the graph $G$ is clear from the context).
For example, $E_R[X, \bar{X}] = E_G[X, \bar{X}] \cap R$.
Let
\begin{eqnarray*}
\partial f(X) &=& \sum_{v \in X} \partial f(v), \\
 \Theta(X) &=&  k |E_R(X, \bar{X}]| + \partial f(X).
 \end{eqnarray*}
We say the circulation $f$ is {\em balanced} if the following hold:
\begin{itemize}
\item $\sum_{x \in V(G)} \partial f(x) = 0$.
\item For any subset $X$ of $V$, $\Theta(X) \ge k-2.$
\end{itemize}

The special $Z_{2k+1}$-circulation $f$ in $Q$ we construct in the first step will be a balanced circulation.
Lemma \ref{balanced} below shows that such a circulation exists.

\begin{lemma}
\label{balanced}
There exists a balanced special $Z_{2k+1}$-circulation $f$ in $Q$.
\end{lemma}
\begin{proof}
By Claim \ref{clm1}, $R$ is $6k$-edge connected. By Nash-Williams' Theorem (cf. \cite{CQbook}),
$R$ contains $3k$ edge disjoint spanning trees, $T_1, T_2, \ldots, T_{3k}$.
Let $G'$ be the subgraph of $G$ induced by $Q \cup T_1 \cup T_2$
(here $G'$ is an ordinary graph, the signs on the edges are ignored).
It is well-known (cf. \cite{CQbook}) that any spanning tree of $G'$  contains a
{\em parity subgraph} $F$ of $G'$, i.e., for each vertex $x$, $$d_F(x) \equiv d_{G'}(x) \pmod{2}.$$
Let $F$ be a parity subgraph of $G'$ contained in $T_2$. Then $G'-F$ is connected
(as it contains a spanning tree $T_1$) and every vertex has an even degree, and hence has an eulerian cycle  $W$.
We orient the edges in $Q$ as follows: We start from an arbitrary vertex $v_0$,
follow  the eulerian cycle $W$,
label the  edges in $Q$  alternately  source edge and sink edge
(the positive edges in $W$ are ignored in this labeling process).
In other word, assume we traverse the eulerian cycle $W$, the negative edges encountered on the way
are $e_1, e_2, \ldots, e_q$. Then $e_{2i-1}$ are source edges and $e_{2i}$ are sink edges.
In particular, if $|Q|$ is even, then the number of sink edges is the same as the number of source edges.
If $|Q|$ is odd, then the number of sink edges is $1$ less than the number of source edges.
This completes the construction of the orientation $\tau$ of $Q$.

If $|Q|$ is even, then let $f(e)=k$ for every edge $e \in Q$. If $|Q|$ is odd, then $|Q| \ge 2k+1$. Let $T'$ be a
set of $k$ sink edges. Let $f(e)=k+1$ if $e \in T'$ and $f(e) = k$ if $e \in Q \setminus T'$.
This completes the construction of the special $Z_{2k+1}$-circulation $f$ in $Q$.

Now we prove that $f$ is a balanced special $Z_{2k+1}$-circulation.

It follows  from the construction that $\sum_{e \in S} f(e) = \sum_{e\in T} f(e)$.
As each sink (resp. source) edge contributes $2f(e)$ (resp. $-2f(e)$) to $\sum_{x \in V} \partial f(x)$, we have
$$\sum_{x \in V} \partial f(x) = 0.$$

We shall show that for any subset $X$ of $V$, $\Theta(X) \ge k-2$.
Starting from a vertex in $\bar{X}$, we traverse the eulerian cycle $W$.
Each time we enter $X$ and leave $X$, we traverse through a
segment of $W$ contained in $X$, and two edges in $E[X, \bar{X}]$.
Let $(e'_1, e'_2, \ldots, e'_b)$ be such a segment, with $e'_1, e'_b \in E[X, \bar{X}]$,
and $e'_i \in G[X]$ for $i = 2,3,\ldots, b-1$.
Now we calculate the contribution of these edges to $\Theta(X)$.
The following follows from the definition of $\Theta(X)$.
\begin{enumerate}
\item If $e'_i \in G[X]$ is a source (resp. sink) edge, then
$e'_i$ contributes $-2f(e'_i)=-2k$ (resp. $2f(e'_i))\ge 2k$) to $\Theta(X)$.
\item If $e'_i \in G[X]$ is a positive edge, then
$e'_i$ contributes $0$   to $\Theta(X)$.
\item If $i \in \{1, b\}$ and $e'_i$   is a  source (resp. sink) edge, then
$e'_i$ contributes $-k$ (resp. $f(e'_i)\ge k$) to $\Theta(X)$.
\item If $i \in \{1, b\}$ and $e'_i$ is a positive edge, then it contributes
$k$ to $\Theta(X)$.
\end{enumerate}
By our orientation, the negative edges in $W$ are alternately source edge and sink edge, except that in case $|Q|$ is odd,
there are two consecutive source edges (i.e., two source edges   not separated by a sink edge, but maybe separated by some
positive edges).

We claim that the contribution of the edges in this segment to $\Theta(X)$ is non-negative, except that when the segment contains two consecutive
source edges, the contribution of this segment is at least $-2k$.

For the proof of this claim, we need to consider a few cases according to whether $e'_1, e'_b$ are
positive edges, or one positive and the other  is a source edge, or one is positive and the other is a sink edge, etc.
However, each case is straightforward. We just consider two cases, and for simplicity, we assume that the segment
does not contain two consecutive source edges (if it does have two consecutive source edges, then we need to add $-2k$
to the total contribution).

If both $e'_1, e'_b$ are positive edges, then the number of source edges in this segment is
at most one more than the number of sink edges. Since each of $e'_1, e'_b$ contributes $k$ to
$\Theta(X)$, we conclude that the total contribution of this segment to $\Theta(X)$ is non-negative.

If both $e'_1$ and $e'_b$ are sink edges,  the number of  source edges in this segment is
one more than the number of sink edges. Since each of $e'_1, e'_b$ contributes at least $k$ to
$\Theta(X)$, we conclude that the total contribution of this segment to $\Theta(X)$ is non-negative.

Add up the contribution of   all the edges in the eulerian cycle $W$, we conclude that
the total contribution is at most $-2k$, where $-2k$ is contributed by the segment containing two consecutive
source edges.

Now each spanning tree $T_i$ for $i=3,4,\ldots, 3k$ contains at least one edge in
$E_R(X, \bar{X}]$, and hence contributes at least $k$ to
$\Theta(X)$. So $\Theta(X) \ge 3k-2 - 2k = k-2$.
This completes the proof of Lemma \ref{balanced}.
\end{proof}

Let $f$ be a balanced special $Z_{2k+1}$-circulation $f$ in $Q$.

\begin{claim}
There is a special $(2k+1)$-circulation $g$ of $R$ such that  $f+g$  is a special $Z_{2k+1}$-flow in $(G, \sigma)$.
\end{claim}
\begin{proof}
Let $\beta: V(G) \to Z_{2k+1}$ be defined as
$$\beta(x) \equiv 2 \partial f(x) \pmod{2k+1}.$$
Then $\sum_{x \in V(G)} \beta(x) \equiv 0 \pmod{2k+1}$.

Since $R$ is $6k$-edge connected,
by Theorem \ref{mod-o}, $R$ has a $\beta$-orientation $D$.
Let $g(e) = k$ for $e \in R$. Then for each vertex $x$,
\begin{eqnarray*}
\partial g(x) &\equiv& k\beta(x) \pmod{2k+1} \\
&\equiv &  2k \partial f(x) \pmod{2k+1} \\
&\equiv&  - \partial f(x) \pmod{2k+1}.
\end{eqnarray*}
Hence $\partial (f+g)(x) = \partial f(x)+ \partial g(x) \equiv 0 \pmod{2k+1}$ for each vertex $x$, i.e.,  $f+g$ is a
special $Z_{2k+1}$-flow in $(G, \sigma)$.
\end{proof}

For a $Z_{2k+1}$-flow $\phi$ in $(G, \sigma)$, let
$$||\phi||= \sum_{x \in V(G)} | \partial \phi (x)|.$$
Thus a special $Z_{2k+1}$-flow $\phi$ is a $(2k+1,k)$-flow if and only if $||\phi||=0$.

Among all the special $Z_{2k+1}$-flows $\phi$ in $(G, \sigma)$  of the form $f+g$, choose one  for which
$||(f+g)||$
is minimum. If $||(f+g)|| = 0$, then $f+g$ is a $(2k+1,k)$-flow in $(G, \sigma)$, and we are done.

Assume this is not the case.
Let $V^+=\{x:   \partial (f+g) (x)> 0\}$ and $V^-=\{x: \partial (f+g) (x) < 0\}$.
Since $\sum_{x \in V(G)}  \partial (f+g) (x) =0$, $V^+ \ne \emptyset$ and $V^- \ne \emptyset$.
Let $D$ be the orientation of $R$ associated with the circulation $g$.
We say a vertex $y$ of $G$ is {\em reachable} if there is a directed path in $D$ from $y$ to a vertex $x \in V^-$.
In particular, every vertex in $V^-$ is reachable.
Let $Y$ be the set of all reachable vertices.

Assume first that $Y \cap V^+ \ne \emptyset$. Let $P$ be a directed path in $D$ from $y \in V^+$ to $x \in V^-$.
We reverse the orientation of the edges in $P$, and let $g'(e) = 2k+1-g(e)$ for $e \in P$ and
$g'(e) = g(e)$ for $e \notin P$. Then $f+g'$ is a special $Z_{2k+1}$-flow in $(G,\sigma)$ with
\begin{equation*}
 \partial (f+g') (v)=\begin{cases}
 \partial (f+g) (v) - (2k+1),  &  \text{if $v = y$,}
\\
\partial(f+g) (v) + (2k+1),  & \text{if $v = x$,}
\\
\partial (f+g) (v),  & \text{otherwise.}
\end{cases}
\end{equation*}
As $f+g$ is a $Z_{2k+1}$-flow, $\partial (f+g) (x) < 0$ and $\partial (f+g) (y) > 0$  imply  that $\partial (f+g) (x) = -a(2k+1)$ for some positive integer $a$,
and $\partial (f+g) (y) = b(2k+1)$ for some positive integer $b$. Therefore $||(f+g')|| = ||(f+g)|| - 2(2k+1)$,
contrary to our choice of $g$.

Assume $Y \cap V^+ = \emptyset$. Since $X^- \subseteq Y$, we have $\sum_{v \in Y}\partial (f+g) (v) \le -(2k+1)$.
If there exist $y' \in \bar{Y}$ and $y \in Y$ such that $(y',y)$ is a directed edge of $D$, then  $y'$ would be a
reachable vertex, a contradiction.
Thus all edges in $E_D[Y, \bar{Y}]$ are oriented
from $Y$ to  $\bar{Y}$. Observe that
$\sum_{v \in Y}\partial (f+g) (v) = \sum_{v \in Y}\partial f (v) +\sum_{v \in Y}\partial g (v)$.
Since each edge in $E_D[Y, \bar{Y}]$ contribute $k$ to  $\sum_{v \in Y}\partial g (v)$, and every other edge
contributes $0$ to  $\sum_{v \in Y}\partial g (v)$, we conclude that
$\sum_{v \in Y}\partial (f+g) (v) = \Theta(Y)$. By Lemma \ref{balanced}, $\Theta(Y) \ge k-2$, contrary to the
conclusion that $\sum_{v \in Y}\partial (f+g) (v) \le -(2k+1)$.

This completes the proof of Theorem \ref{main1}.
 \end{proof}

\bigskip
Jaeger's $(2+\frac 1k)$-flow conjecture is sharp: there are $(4k-1)$-edge connected graphs $G$ for which $\Phi_c(G) > 2  +\frac 1k$.

Corresponding to Jaeger's conjecture, one naturally wonder what is the least integer $\psi(k)$ such that every essentially $(2k+1)$-unbalanced $\psi(k)$-edge connected signed graph $(G,\sigma)$ have
$\Phi_c(G, \sigma) \le 2 + \frac 1k$? The current known bounds are $4k \le \psi(k) \le 12k-1$. It would be interesting to narrow the gap between the upper and lower bounds.

\bibliographystyle{plain}

\begin{thebibliography}{10}
\bibitem{Bou83}
A.~Bouchet,
\newblock {\em Nowhere-zero integral flows on bidirected graph},
\newblock J. Combin. Theory Ser. B 34 (1983), 279-292.






\bibitem{GTZ98} L. A. Goddyn, M. Tarsi and C. Q. Zhang, {\em On
$(k,d)$-colorings and fractional nowhere zero flows},  J. Graph
Theory, 28(1998), 155-161.


\bibitem{JAE81}
F.~Jaeger,
\newblock On circular flows in graphs.
\newblock Proceedings of the 6th Hungarian Colloquium on
Combinatorics, Eger, Hungary, 1981.

\bibitem{FJ}
F.~Jaeger,
\newblock {\em Nowhere-zero flow Problems},
\newblock Selected topics in Graph Theory 3
\newblock Academic Press, London 1988,   71-95.
\bibitem{Khell87}
A.~Khelladi,
\newblock {\em Nowhere-zero integral chains and flows in bidirected graphs},
\newblock J. Combin. Theory Ser. B. 43 (1987), 95-115.

\bibitem{kundu}
S.~Kundu,
\newblock {\em Bounds on the number of disjoint spanning trees},
\newblock J. Combin. Theory Ser. B. 17 (1974), 199-203.

\bibitem{LTWZ2012} L. M. Lovasz, C. Thomassen, Y. Wu and C. Q. Zhang,
\newblock {\em Nowhere-zero $3$-flows and modulo $k$-orientations},
\newblock J. Combin. Theory Ser. B., to appear.


\bibitem{nash} C.S.J.A. Nash-Williams, {\em Edge disjoint spanning
trees of finite graphs}, J. London Math. Soc. 36(1961), 445-450.

\bibitem{RZ2011} A. Raspaud and X. Zhu, {\em Circular flow on signed graphs},  J. Combin. Theory Ser. B. 101(2011), no. 6, 490-501.

\bibitem{Thomassen2012} C. Thomassen, {\em The weak $3$-flow conjecture and the waek circular flow conjecture},
J. Combin. Theory Ser. B. 102(2012),  521-529.

\bibitem{Tutte1949} W. T. Tutte, {\em On the embedding of linear graphs in surfaces}, Proc. London Math. Soc., Ser. 2, 51 (1949), 474-483.

\bibitem{Tutte1954}  W. T. Tutte, {\em A construction on the theory of chromatic polynomial}, Canad. J. Math., 6 (1954), 80-91.


\bibitem{XuZhang05} R. Xu and C. Zhang, {\em On flows in
bidirected graphs}, Discrete Mathematics, 299(2005) 335-343.




\bibitem{CQbook} C. Q. Zhang, {\em Integer flows and cycle covers of graphs}, Marcel Dekker, New York. 1997.

\bibitem{Survey01}
X.~Zhu,
\newblock{\em Circular chromaic number - a survey},
\newblock Discrete Mathematics, 229 (2001), no. 1-3, 371-410.


\bibitem{Zhu05}
X.~Zhu,
\newblock{\em Recent developments in circular colouring of graphs}
\newblock Topics in Discrete Mathematics, 2006, 497-550.

\bibitem{Zyka87}
O.~Z\'yka,
\newblock {\em Nowhere-zero 30-flow on bidirected graphs}, Thesis, Charles Universiy, Pragua,1
 87, KAM-DIMATIA Seris 87-26.



\end{thebibliography}

\end{document}